\documentclass[a4paper]{amsart}

\usepackage{enumerate}
\usepackage{xy} \xyoption{all}
\usepackage{aliascnt}
\usepackage{hyperref}
\usepackage{graphicx}

\DeclareMathOperator{\locdim}{locdim}
\DeclareMathOperator{\sr}{sr}
\DeclareMathOperator{\rr}{rr}
\DeclareMathOperator{\xrr}{xrr}
\DeclareMathOperator{\Lg}{Lg}
\DeclareMathOperator{\id}{id}

\newcommand{\subseteqRotatedUp}{\mathrel{\reflectbox{\rotatebox[origin=c]{90}{$\subseteq$}}}}

\newcommand{\andSep}{\,\,\,\text{ and }\,\,\,}
\newcommand{\Bdd}{\mathcal{B}}
\newcommand{\spec}{\mathrm{sp}}
\newcommand{\tuple}[1]{\mathbf{#1}}
\newcommand{\sa}{\text{sa}}
\newcommand{\RR}{{\mathbb{R}}}
\newcommand{\KK}{{\mathcal{K}}}
\newcommand{\NN}{{\mathbb{N}}}
\newcommand{\CC}{{\mathbb{C}}}
\newcommand{\ca}{$C^*$-algebra}

\numberwithin{equation}{section}

\newtheorem{lma}{Lemma}[section]

\newaliascnt{thmCt}{lma}
\newtheorem{thm}[thmCt]{Theorem}
\aliascntresetthe{thmCt}

\newaliascnt{corCt}{lma}
\newtheorem{cor}[corCt]{Corollary}
\aliascntresetthe{corCt}

\newaliascnt{prpCt}{lma}
\newtheorem{prp}[prpCt]{Proposition}
\aliascntresetthe{prpCt}

\theoremstyle{definition}

\newaliascnt{pgrCt}{lma}
\newtheorem{pgr}[pgrCt]{}
\aliascntresetthe{pgrCt}

\newaliascnt{dfnCt}{lma}
\newtheorem{dfn}[dfnCt]{Definition}
\aliascntresetthe{dfnCt}

\newaliascnt{rmkCt}{lma}
\newtheorem{rmk}[rmkCt]{Remark}
\aliascntresetthe{rmkCt}

\newaliascnt{rmksCt}{lma}
\newtheorem{rmks}[rmksCt]{Remarks}
\aliascntresetthe{rmksCt}

\newaliascnt{qstCt}{lma}
\newtheorem{qst}[qstCt]{Question}
\aliascntresetthe{qstCt}

\newaliascnt{exaCt}{lma}
\newtheorem{exa}[exaCt]{Example}
\aliascntresetthe{exaCt}

\newcounter{theoremintro}

\newaliascnt{thmIntroCt}{theoremintro}
\newtheorem{thmIntro}[thmIntroCt]{Theorem}
\aliascntresetthe{thmIntroCt}

\title{Real rank of extensions of C*-algebras}

\author{Hannes Thiel}
\address{Hannes~Thiel, 
Department of Mathematical Sciences, Chalmers University of Technology and University of
Gothenburg, Gothenburg SE-412 96, Sweden.}
\email{hannes.thiel@chalmers.se}
\urladdr{www.hannesthiel.org}

\subjclass[2010]%
{Primary
46L05, 
46L85; 
Secondary
46L45, 
46M20. 
}

\keywords{$C^*$-algebras, dimension theory, stable rank, real rank}

\thanks{
The author was partially supported by the Knut and Alice Wallenberg Foundation (KAW 2021.0140).
}

\date{\today}

\begin{document}

\begin{abstract}
Given a closed ideal $A$ in a \ca{} $E$, we develop techniques to bound the real rank of $E$ in terms of the real ranks of $A$ and $E/A$.
Building on work of Brown, Lin and Zhang, we obtain explicit computations if $A$ belongs to any of the following classes: 
(1)~\ca{s} with real rank zero, stable rank one and vanishing $K_1$-group;
(2)~simple, purely infinite \ca{s};
(3)~simple, $\mathcal{Z}$-stable \ca{s} with real rank zero;
(4)~separable, stable \ca{s} with an approximate unit of projections and the Corona Factorization Property.
\end{abstract}

\maketitle

\section{Introduction}

It is a classical result that a compact, Hausdorff space $X$ has covering dimension at most $n$ if and only if every continuous function $X \to \RR^{n+1}$ can be approximated by a continuous function $g \colon X \to \RR^{n+1}$ such that $0 \notin g(X)$.
Transferring this characterization to the setting of noncommutative spaces, Brown and Pedersen \cite{BroPed91CAlgRR0} defined the \emph{real rank} of a unital \ca{} $A$ to be at most $n$ if every tuple of $n+1$ self-adjoint elements in $A$ can be approximate arbitrarily closely by a tuple $(b_0,\ldots,b_n) \in A_\sa^{n+1}$ such that $\sum_j b_j^2$ is invertible;
see \cite[Proposition~3.3.2]{Pea75DimThy} and \autoref{dfn:rr} below.

Similar to the covering dimension of topological spaces, it is of interest to compute the real rank of \ca{s}.
But while the theory of covering dimension is well developed and has many tools available for computations, the situation is much murkier for the real rank.
Only few general methods are known, and in some situations, the behavior of the real rank substantially differs from that of the covering dimension.

For example, if $X$ and $Y$ are compact, Hausdorff spaces, then their product satisfies $\dim(X \times Y) \leq \dim(X) + \dim(Y)$; 
see \cite[Proposition~3.2.6]{Pea75DimThy}.
Analogously, one would expect that for the minimal tensor product of unital \ca{s}, one has $\rr(A \otimes B) \leq \rr(A) + \rr(B)$, but it is known that this can fail even if $A$ and $B$ have real rank zero.
For example, the algebra $\Bdd(H)$ of bounded linear operators on a separable, infinite-dimensional Hilbert space $H$ has real rank zero, but $\Bdd(H) \otimes \Bdd(H)$ does not;
see \cite{Osa99NonzeroRR}.
The exact value of the real rank of $\Bdd(H) \otimes \Bdd(H)$ is unknown.
There even exist examples of nuclear \ca{s} $A$ and $B$ of real rank zero such that $A \otimes B$ has nonzero real rank;
see \cite{KodOsa95RRTensProd, KodOsa01FS}.

\medskip

In this paper we develop methods to compute the real rank of an extension 
\[
0 \to A \to E \to B \to 0
\]
of \ca{s}.
(This means that $E$ is a \ca{} containing an ideal $A$ such that $E/A \cong B$.)
We aim to determine the real rank of $E$ in terms of those of~$A$ and~$B$.
Since the real rank of a nonunital \ca{} is defined as that of its minimal unitization, the real rank should more accurately be thought of as a noncommutative analog of local covering dimension;
see \cite[Section~2.2(ii)]{BroPed09Limits}.
For $\sigma$-compact, locally compact, Hausdorff spaces $X$, we have $\locdim(X) = \dim(X)$, but in general the local covering dimension may be strictly smaller than the covering dimension.
We refer to \cite[Chapter~5]{Pea75DimThy} for details.

Given a locally compact, Hausdorff space $X$ and an open subset $U \subseteq X$, one can deduce from Propositions~3.5.6, 5.2.1 and~5.2.2 in \cite{Pea75DimThy} that $\locdim(X) = \max\{ \locdim(U), \locdim(X\setminus U)\}$.
Analogously, one might expect that for an extension $0 \to A \to E \to B \to 0$ of \ca{s} one has $\rr(E) = \max\{ \rr(A), \rr(B) \}$, and in many situations this does indeed hold -- for example, if $E$ is a CCR algebra (\cite[Theorem~3.6]{Bro16HigherRRSR}), or if $A$ has only finite-dimensional irreducible representations (\cite[Theorem~3.11]{Bro16HigherRRSR}).

However, it is also known that the identity $\rr(E) = \max\{ \rr(A), \rr(B) \}$ can fail even if $A$ and $B$ are nuclear and have real rank zero.
In general, it is known that the real rank does not increase when passing to ideals or quotients, which implies that $\max\{ \rr(A), \rr(B) \} \leq \rr(E)$;
see \cite[Th\'{e}or\`{e}me~1.4]{Elh95RRExt}.
Our first result gives a general upper bound for $\rr(E)$:

\begin{thmIntro}[{\ref{prp:rrExt}}]
\label{thmA}
Given an extension $0 \to A \to E \to B\to 0$ of \ca{s}, we have
\[
\max\{\rr(A),\rr(B)\}
\leq \rr(E) 
\leq \rr(A) + \rr(B) + 1.
\]
\end{thmIntro}

An interesting consequence of \autoref{thmA} is that $\rr(E)$ is restricted to only two possible values whenever $A$ or $B$ has real rank zero.
In particular, if $\rr(A)=\rr(B)=0$, then $\rr(E) \in \{0,1\}$, and it is known that $\rr(E)=0$ is equivalent to the $K$-theoretic condition that the index map $K_0(B) \to K_1(A)$ vanishes.
This leads to a computation of the real rank of an extension of real rank zero \ca{s};
see \autoref{prp:rrExt-rrzero}.
We use this to show that the stable multiplier algebra of the Calkin algebra has real rank one;
see \autoref{exa:calkin}.

In general it is not known where in the range between $\max\{\rr(A),\rr(B)\}$ and $\rr(A) + \rr(B) + 1$ the real rank of $E$ can fall.
I am not aware of any example where $\rr(E) \geq \max\{\rr(A),\rr(B)\} + 2$.

A low real rank can be considered as a regularity property, and real rank zero is undoubtedly the most important since it has numerous consequences for the structure and classification theory of such \ca{s}.
But other nonzero low values of the real rank are also relevant.
For example, if a \ca{} has real rank at most one, then its full elements are dense, and this is of interest in connection to the study of the Global Glimm Problem \cite{ThiVil23Glimm, ThiVil24SoftOps, AsaThiVil23arX:RksSoftOps}.

\medskip

To improve the upper bound in \autoref{thmA}, we introduce the \emph{extension real rank} $\xrr(A)$ for a \ca{} $A$.
This invariant captures whether tuples of self-adjoint elements in the multiplier algebra of $A$ can be perturbed in a certain way;
see \autoref{dfn:xrr}.
We prove:

\begin{thmIntro}[{\ref{prp:rrExt-xrr-rr}}]
\label{thmB}
Given an extension $0 \to A \to E \to B\to 0$ of \ca{s}, we have
\[
\max \big\{ \rr(A), \rr(B) \big\}
\leq \rr(E)
\leq \max \big\{ \xrr(A), \rr(B) \big\}.
\]
\end{thmIntro}

Thus, if $\xrr(A) \leq 1$, then $\rr(E) \leq \max\{1,\rr(B)\}$, and we can completely determine $\rr(E)$ in terms of $\rr(B)$, $\rr(A)$ and $K$-theoretic information;
see \autoref{prp:ComputeWithLowXRR}.

We prove that the extension real rank of a \ca{} is bounded by the real rank of its multiplier algebra;
see \autoref{prp:rr-xrr-rrM}.
We also show that the extension real rank enjoys certain permanence properties that the real rank of the multiplier algebra does not have;
see \autoref{sec:Permanence}.
Using these methods, we verify in \autoref{sec:LowXRR} that large classes of \ca{s} have extension real rank at most one, which in turn allows us to completely determine the real rank of extensions by such algebras;
see Propositions~\ref{prp:Classes-xrr0} and~\ref{prp:Classes-xrr1}.

\subsection*{Conventions}

We use $\NN$ to denote the natural numbers including $0$. 
Given a \ca{} $A$, we let $A_\sa$ denote its set of self-adjoint elements.
We let $\widetilde{A}$ denote the minimal unitization of $A$.
Note that $\widetilde{A}=A$ if $A$ is unital.
We will also consider the forced unitization of~$A$, which agrees with $\widetilde{A}$ whenever $A$ is nonunital, and which is given as $A \oplus \CC$ whenever $A$ is unital.

By an ideal in a \ca{} we mean a closed, two-sided ideal, and by a morphism between \ca{s} we mean a $\ast$-homomorphism.
We use $\KK$ to denote the \ca{} of compact operators on a separable, infinite-dimensional Hilbert space.
We let $\spec(a)$ denote the spectrum of an operator $a$.

\subsection*{Acknowledgements}

The author thanks Eusebio Gardella for valuable feedback on a first version of this paper, and Ilijas Farah for valuable comments regarding the L\"{o}wenheim-Skolem condition for the real rank.

The author also thanks the anonymous referee for beneficial feedback.

\section{The real rank of extensions} 
\label{sec:rrExt}

In this section, we first recall the definition and some basic properties of the real rank.
Given an extension $0 \to A \to E \to B \to 0$ of \ca{s}, we prove the folklore result that the real rank of $E$ is bounded by the maximum of the real rank of $B$ and of the multiplier algebra $M(A)$;
see \autoref{prp:rrExtUsingMult}.
We also show that $\rr(E) \leq \rr(A)+\rr(B)+1$;
see \autoref{prp:rrExt}.
This allows us to compute the real rank of extensions of real rank zero \ca{s};
see \autoref{prp:rrExt-rrzero}.

\medskip

Let $A$ be a unital \ca{}.
We use $A^{-1}$ to denote the set of invertible elements in $A$.
Set
\[
\Lg_n(A)_\sa
:= \left\{ (a_1,\ldots,a_n)\in (A_\sa)^n \, \middle| \, \sum_{k=1}^n a_k^2\in A^{-1} \right\}.
\]
The abbreviation `Lg' stands for `left generators', and the reason is that a tuple $(a_1,\ldots,a_n)\in A^n_\sa$ belongs to $\Lg_n(A)_\sa$ if and only if $\{a_1,\ldots,a_n\}$ generates $A$ as a left (not necessarily closed) ideal, that is, $Aa_1+\ldots+Aa_n=A$.
A tuple in $(A_\sa)^n$ is said to be \emph{unimodular} if it belongs to $\Lg_n(A)_\sa$.

\begin{dfn}[{Brown-Pedersen, \cite{BroPed91CAlgRR0}}]
\label{dfn:rr}
The \emph{real rank} of a unital \ca{}~$A$, denoted by $\rr(A)$, is the least integer $n\geq 0$ such that $\Lg_{n+1}(A)_\sa$ is dense in $A_\sa^{n+1}$.
We set $\rr(A)=\infty$ if $\Lg_{n+1}(A)_\sa$ is not dense in~$A_\sa^{n+1}$ for any $n$.

If $A$ is a nonunital \ca{}, then $\rr(A):=\rr(\widetilde{A})$. 
\end{dfn}

\begin{rmk}
\label{rmk:Lg-dense}
Let $A$ be a unital \ca, and let $n \geq 1$.
If $\Lg_n(A)_\sa$ is dense in $A^n_\sa$, then it easily follows that $\Lg_{n+1}(A)_\sa$ is dense in $A^{n+1}_\sa$.
This should be contrasted with \autoref{rmk:xrr-Ln}.
\end{rmk}

We will use the following folklore result, which was noted in \cite[Remark~3.13]{Bro16HigherRRSR}.

\begin{prp}
\label{prp:rrExtUsingMult}
Let $0 \to A \to E \to B \to 0$ be an extension of \ca{s}.
Then
\[
\rr(E) \leq \max\big\{ \rr(M(A)), \rr(B) \big\}. 
\]
\end{prp}
\begin{proof}
By \cite[Proposition~1.6]{NagOsaPhi01RanksAlgCts}, if a \ca{} $C$ is the pullback along a morphism $F \to H$ and a surjective morphism $G \to H$, then the real rank of $C$ is dominated by the maximum of $\rr(F)$ and $\rr(G)$.
(Note that using forced unitizations one can remove the unitality assumptions in \cite[Proposition~1.6]{NagOsaPhi01RanksAlgCts}.)

Now the result follows using the standard result of Busby that $E$ can be realized as the pullback along the natural quotient map $M(A) \to M(A)/A$ and the Busby map $B \to M(A)/A$.
\end{proof}

The next result provides rough bounds for the real rank of an extension.

\begin{thm}
\label{prp:rrExt}
Let $0 \to A \to E \to B \to 0$ be an extension of \ca{s}.
Then
\[
\max\big\{ \rr(A),\rr(B) \big\} 
\leq \rr(E) 
\leq \rr(A) + \rr(B) + 1.
\]
\end{thm}
\begin{proof}
By \cite[Th\'{e}or\`{e}me~1.4]{Elh95RRExt}, the real rank does not increase when passing to ideals or quotients.
This implies the first inequality.

Let us show the second inequality.
We identify $A$ with an ideal in $E$ such that~$B$ is naturally isomorphic to~$E/A$, and we let $\pi\colon E \to B$ denote the quotient map.
If~$E$ is nonunital, then we consider $A$ as an ideal of $\widetilde{E}$.
Then $\widetilde{E}/A$ is isomorphic to the forced unitization of $B$, whence $\rr(\widetilde{E}/A)=\rr(B)$.
It then suffices to show that $\rr(\widetilde{E})\leq\rr(A)+\rr(\widetilde{E}/A)$.
Thus, we may assume from now on that $E$ is unital.

Set $k:=\rr(A)+1$ and $l:=\rr(B)+1$, which we may assume are finite.
To show that $\Lg_{k+l}(E)_\sa$ is dense in $E^{k+l}_\sa$, let $\varepsilon>0$ and let
\[
(x_1,\ldots,x_k,y_1,\ldots,y_l)\in E^{k+l}_\sa.
\]
be given.
Set
\[
b_1 := \pi(y_1), \ldots, b_l:=\pi(y_l).
\]
Using that $\rr(B)\leq l-1$, we find $(b_1',\ldots,b_l')\in\Lg_l(B)_\sa$ such that $\|b_j-b_j'\|<\varepsilon$ for each $j$.
Choose $c_1,\ldots,c_l\in B$ such that $c_1b_1'+\ldots+c_lb_l'=1_B$.
Using that self-adjoint elements in quotients can be lifted to self-adjoint elements, and using standard properties of the quotient norm, we can choose $y_1',\ldots,y_l'\in E_\sa$ and $z_1,\ldots,z_l\in E$ such that
\[
\pi(y_j')=b_j', \andSep
\pi(z_j)=c_j, \andSep
\|y_j-y_j'\|<\varepsilon \quad \text{ for $j=1,\ldots,l$.}
\]

Set $e:=z_1y_1'+\ldots+z_ly_l' \in E$, and set
\[
d_1 := x_1 - e^*x_1e,
\quad \ldots, \quad
d_k := x_k - e^*x_ke.
\]
Note that $\pi(e)=1_B$, and therefore the elements $d_1,\ldots,d_k$ belong to $A_\sa$.
Using that $\rr(A)\leq k-1$, we find $(d_1',\ldots,d_k')\in\Lg_k(\tilde{A})_\sa$ such that $\|d_i-d_i'\|<\varepsilon$ for each~$i$.
By regarding $\widetilde{A}$ as a unital subalgebra of $E$, we view each $d_i'$ as an element in~$E_\sa$.
Then $(d_1',\ldots,d_k')\in\Lg_k(E)_\sa$ since $\sum_i (d_i')^2$ is invertible in $\widetilde{A}$ and therefore also in $E$.
Set
\[
x_1' := d_1' + e^*x_1e,
\quad \ldots, \quad
x_k' := d_k' + e^*x_ke.
\]
For each $i$, we have
\[
\| x_i-x_i' \|
= \big\| (d_i + e^*x_ie) - (d_1' + e^*x_1e) \big\|
= \| d_i - d_i' \|
< \varepsilon.
\]
It remains to verify that $(x_1',\ldots,x_k',y_1',\ldots,y_l')$ belongs to $\Lg_{k+l}(E)_\sa$.
Let $L$ be the left ideal generated by $x_1',\ldots,x_k',y_1',\ldots,y_l'$.
By construction, we have $e\in L$.
Hence, for each $i$ we obtain $e^*x_ie\in L$, and then $d_i'\in L$.
Since $d_1',\ldots,d_k'$ generate~$E$ as a left ideal, we obtain $L=E$.
\end{proof}

It follows immediately from \autoref{prp:rrExt} that $\rr(E) \leq 1$ if $E$ is an extension of real rank zero \ca{s}.
To complete the picture, we recall the characterization \cite[Proposition~4]{LinRor95ExtLimitCircle} of when real rank zero passes to extensions.
This result has also (partially) been obtained in \cite[Theorem~3.14]{BroPed91CAlgRR0} and \cite[Lemma~2.4]{Zha92RR0CoronaMultiplier1}.

\begin{prp}[{\cite[Proposition~4]{LinRor95ExtLimitCircle}}]
\label{prp:rrzeroExt}
Let $0 \to A \to E \to B \to 0$ be an extension of \ca{s}.
Assume that $\rr(A)=\rr(B)=0$.
Then the following are equivalent:
\begin{enumerate}
\item
We have $\rr(E)=0$.
\item
The index map $K_0(B)\to K_1(A)$ vanishes.
\item
The natural map $K_0(E)\to K_0(B)$ is surjective.
\item
Every projection in $B$ can be lifted to a projection in $E$.
\end{enumerate}
\end{prp}
\begin{proof}
The equivalence of~(1), (2) and~(4) is shown in \cite[Proposition~4]{LinRor95ExtLimitCircle}.
The equivalence of~(2) and ~(3) follows from the $K$-theory six-term exact sequence;
see \cite[Corollary~V.1.2.22]{Bla06OpAlgs}.
\end{proof}

\begin{cor}
\label{prp:rrExt-rrzero}
Let $0 \to A \to E \to B \to 0$ be an extension of \ca{s}.
Assume that $\rr(A)=\rr(B)=0$.
Then:
\[
\rr(E)= \begin{cases}
0, & \text{if the index map $K_0(B)\to K_1(A)$ vanishes}; \\
1, & \text{otherwise}.
\end{cases}
\]

Further, if $F \subseteq E$ is a hereditary sub-\ca{}, then $\rr(F) \leq \rr(E)$.
\end{cor}
\begin{proof}
It follows from \autoref{prp:rrExt} that $\rr(E) \leq 1$.
Hence, the computation of $\rr(E)$ follows from \autoref{prp:rrzeroExt}.
Let $F \subseteq E$ be a heredtiary sub-\ca{}.
If $\rr(E)=0$, then $\rr(F)=0$ since real rank zero passes to hereditary sub-\ca{s}; 
see \cite[Corollary~2.8]{BroPed91CAlgRR0}.
Thus, it suffices to show that $\rr(F) \leq 1$.
This follows since $F$ is an extension of the real rank zero \ca{s} $F/(F \cap A)$ and $F \cap A$.
\end{proof}

\begin{rmk}
If a \ca{} has real rank zero, then so does every hereditary sub-\ca{};
see \cite[Corollary~2.8]{BroPed91CAlgRR0}.
In \autoref{prp:rrExt-rrzero} we exhibit another situation where the real rank does not increase when passing to a hereditary sub-\ca{}.

This does not hold in general, however.
For example, a commutative \ca{} $C(X)$ can be viewed as a (full) hereditary sub-\ca{} in $M_n(C(X))$, and by \cite[Corollary~3.2]{BegEva91RRMatrixValued} we have
\[
\rr(C(X)) = \dim(X), \andSep
\rr(M_n(C(X))) = \left\lceil \frac{\dim(X)}{2n-1} \right\rceil.
\]

More drastically, given any \ca{} $A$ with $\rr(A)=\infty$, we can view~$A$ as a (full) hereditary sub-\ca{} in $A \otimes \KK$, and it follows from \cite[Proposition~3.3]{BegEva91RRMatrixValued} and \cite[Corollary~2.8]{BroPed91CAlgRR0} that 
 $\rr(A \otimes \KK)=1$.
\end{rmk}

\begin{exa}
\label{exa:calkin}
Let $Q$ denote the Calkin algebra.
Then the stable multiplier algebra of $Q$ satisfies
\[
\rr( M(Q\otimes\KK) )=1.
\]
Indeed, $M(Q\otimes\KK)$ contains $Q\otimes\KK$ as a closed ideal, and we have $\rr(Q\otimes\KK)=\rr(M(Q\otimes\KK)/Q\otimes\KK)=0$ and $\rr(M(Q\otimes\KK))\neq 0$ by \cite[Example~2.7(iii)]{Zha92RR0CoronaMultiplier1}.
Thus the result follows from \autoref{prp:rrExt-rrzero}.
\end{exa}

\section{The extension real rank} 

In \autoref{prp:rrExt} we have seen lower and upper bounds for the real rank of an extension.
In this section, we develop a method to obtain better upper bounds.

Given a \ca{} $A$, we introduce the \emph{extension real rank} $\xrr(A)$;
see \autoref{dfn:xrr}.
The main feature of this invariant is that it satisfies
\[
\rr(E) \leq \max \big\{ \xrr(A),\rr(B) \big\}
\]
for every extension $0 \to A \to E \to B \to 0$ of \ca{s};
see \autoref{prp:rrExt-xrr-rr}.
We can view this as a sophisticated version of \autoref{prp:rrExtUsingMult} since $\xrr(A) \leq \rr(M(A))$ by \autoref{prp:rr-xrr-rrM}.
The advantage of $\xrr(A)$ over $\rr(M(A))$ stems from the much better permanence properties of the extension real rank, which will be shown in \autoref{sec:Permanence}.

\begin{dfn}
\label{dfn:xrr}
Let $A$ be a \ca{}, let $n\in\NN$, and let $\pi_A\colon M(A)\to M(A)/A$ denote the quotient map.
We say that $A$ has property $(\Lambda_n)$ if for every tuple $(a_0,\ldots,a_n)\in M(A)_\sa^{n+1}$ satisfying $(\pi_A(a_0),\ldots,\pi_A(a_n))\in\Lg_{n+1}(M(A)/A)_\sa$, and every $\varepsilon>0$, there exists $(b_0,\ldots,b_n)\in\Lg_{n+1}(M(A))_\sa$ such that
\[
\|a_0-b_0\|<\varepsilon, \ldots, \|a_n-b_n\|<\varepsilon, \andSep
\pi_A(a_0)=\pi_A(b_0), \ldots, \pi_A(a_n)=\pi_A(b_n).
\]

The \emph{extension real rank} of $A$, denoted by $\xrr(A)$, is the smallest $n\in\NN$ such that~$A$ has $(\Lambda_m)$ for all $m\geq n$.
(As usual, we set $\xrr(A)=\infty$ if there is no $n\in\NN$ such that $A$ has $(\Lambda_m)$ for all $m\geq n$.)
\end{dfn}

\begin{rmk}
\label{rmk:xrr}
The \emph{extension real rank} of a \ca{} $A$ encodes a property of the universal extension $0 \to A \to M(A) \to M(A)/A \to 0$.
We show in \autoref{prp:LnGeneralExt} below that it leads to an analogous property for \emph{every} extension $0 \to A \to E \to B \to 0$.
Nevertheless, we stress that the extension real rank of $A$ is an invariant depending only on $A$, and not a particular extension involving $A$.
\end{rmk}

\begin{prp}
\label{prp:xrrUnital}
Let $A$ be a unital \ca, and let $n \in \NN$.
Then the following are equivalent:
\begin{enumerate}
\item
The set $\Lg_{n+1}(A)_\sa$ is dense in $A^{n+1}_\sa$.
\item
The \ca{} $A$ has $(\Lambda_n)$.
\item
The \ca{} $A$ has $(\Lambda_m)$ for all $m\geq n$.
\end{enumerate}
In particular, we have $\xrr(A)=\rr(A)$ whenever $A$ is unital.
\end{prp}
\begin{proof}
The equivalence of~(1) and~(2) follows directly from the definition of $(\Lambda_n)$.
Further, by the definition of the real rank, (1) means precisely that $\rr(A) \leq n$.
As noted in \autoref{rmk:Lg-dense}, if $\Lg_{n+1}(A)_\sa$ is dense in~$A^{n+1}_\sa$, then $\Lg_{m+1}(A)_\sa$ is dense in~$A^{m+1}_\sa$ for every $m \geq n$, which shows that~(1) and~(3) are equivalent.
\end{proof}

As we will see in \autoref{exa:StableCalkin}, there are nonunital \ca{s} for which the extension real rank and the real rank do not agree.
In general the real rank is dominated by the extension real rank;
see \autoref{prp:rrLeqLn}.

\begin{rmk}
\label{rmk:xrr-Ln}
For nonunital \ca{s}, it is unclear if $(\Lambda_n)$ implies $(\Lambda_{n+1})$.
\end{rmk}

\begin{lma}
\label{prp:LnGeneralExt}
Let $0 \to A \to E \xrightarrow{\pi} B\to 0$ be an extension of \ca{s} with~$E$ unital, let $n\in\NN$ be such that $A$ has~$(\Lambda_n)$, let $(a_0,\ldots,a_n)\in E_\sa^{n+1}$ be such that the tuple $(\pi(a_0),\ldots,\pi(a_n))$ belongs to $\Lg_{n+1}(B)_\sa$, and let $\varepsilon>0$.
Then there exists $(b_0,\ldots,b_n)\in\Lg_{n+1}(E)_\sa$ satisfying
\[
\| a_j - b_j \| < \varepsilon, \andSep
\pi(a_j) = \pi(b_j) \quad \text{for $j=0,\ldots,n$}.
\]
\end{lma}
\begin{proof}
Let $\pi_A\colon M(A)\to M(A)/A$ denote the quotient map.
From Busby's theory we obtain unital morphisms $\sigma\colon E\to M(A)$ and $\tau\colon B\to M(A)/A$ such that $E$ is the pullback of $B$ and $M(A)$ along $\tau$ and $\pi_A$;
see for example \cite[II.8.4.4]{Bla06OpAlgs}.
This means that the map $E \to M(A) \oplus B$, $a \mapsto (\sigma(a),\pi(a))$, is an injective morphism that identifies $E$ with the sub-\ca{} $F$ of $M(A) \oplus B$ given by
\[
F := \big\{ (x,y) \in M(A) \oplus B : \pi_A(x) = \tau(y) \big\}.
\]

The situation is shown in the following commutative diagram.
\[
\xymatrix@R-5pt{
0 \ar[r] & A \ar[r] \ar[d]^{\id_A} & E \ar[r]^{\pi} \ar[d]^{\sigma} & B \ar[r] \ar[d]^{\tau} & 0 \\
0 \ar[r] & A \ar[r] & M(A) \ar[r]^-{\pi_A} & M(A)/A \ar[r] & 0.
}
\]

Consider $(\sigma(a_0),\ldots,\sigma(a_n)) \in M(A)^{n+1}_\sa$.
The tuple $(\pi_A(\sigma(a_0)),\ldots,\pi_A(\sigma(a_n)))$ is unimodular.
Using that $A$ has $(\Lambda_n)$, we obtain $(b_0',\ldots,b_n') \in\Lg_{n+1}(M(A))_\sa$ satisfying
\[
\| \sigma(a_j) - b_j' \|<\varepsilon, \andSep
\pi_A( b_j' ) = \pi_A(\sigma(a_j)) = \tau(\pi(a_j)) \quad \text{for $j=0,\ldots,n$}.
\]

For each $j=0,\ldots,n$, the pair $(b_j',\pi(a_j))$ belongs to $F$, and we let $b_j$ be the corresponding element in the pullback, that is, $b_j \in E$ is the unique element satisfying $\sigma(b_j) = b_j'$ and $\pi(b_j) = \pi(a_j)$.
We show that $(b_0,\ldots,b_n)$ has the claimed properties.
First, we have
\[
\| a_j - b_j \|
= \max\big\{ \| \pi(a_j - b_j) \|, \| \sigma(a_j - b_j) \| \big\}
= \max\big\{ 0, \| \sigma(a_j) - b_j' \| \big\}
< \varepsilon.
\]

To show that $(b_0,\ldots,b_n)$ is unimodular, let $L$ denote the left ideal of $A$ generated by $b_0,\ldots,b_n$.
Using that $(\pi(a_0),\ldots,\pi(a_n))$ is unimodular in $B$, we deduce that $\pi(L)=B$.
Further, since $(b_0',\ldots,b_n')$ is unimodular in $M(A)$, we obtain $x_0,\ldots,x_n \in M(A)$ such that $x_0b_0'+\ldots+x_nb_n' = 1_{M(A)}$.
Given $a \in A$, consider the elements $ax_0,\ldots,ax_n \in A \subseteq E$.
Then
\[
\sigma\big( (ax_0)b_0 + \ldots + (ax_n)b_n \big)
= (ax_0)b_0' + \ldots + (ax_n)b_n'
= a \in A \subseteq M(A).
\]
Since $\sigma$ identifies the subalgebra $A$ of $E$ with the subalgebra $A$ of $M(A)$, it follows that $a = (ax_0)b_0 + \ldots + (ax_n)b_n \in L$.
Thus, we have $A \subseteq L$, and together $L = E$, as desired.
\end{proof}

\begin{prp}
\label{prp:rrExt-Ln-rr}
Let $0 \to A \to E \to B\to 0$ be an extension of \ca{s}, and let $n\in\NN$.
Assume that $A$ has $(\Lambda_n)$ and that $\rr(B)\leq n$.
Then $\rr(E)\leq n$.
\end{prp}
\begin{proof}
If $E$ is nonunital, we consider $A$ as an ideal of $\widetilde{E}$.
Then $\widetilde{E}/I$ is isomorphic to the forced unitization of $B$, whence $\rr(\widetilde{E}/A)=\rr(B)\leq n$.
It then suffices to show that $\rr(\widetilde{E})\leq n$, since this implies that $\rr(E)\leq n$.

Thus, without loss of generality, we may assume that $E$ is unital.
Let $\pi\colon E\to B$ denote the quotient map.
To verify $\rr(E)\leq n$, let $(a_0,\ldots,a_n)\in E_\sa^{n+1}$ and $\varepsilon>0$. 
We need to find $(c_0,\ldots,c_n)\in\Lg_{n+1}(E)_\sa$ such that $\| a_j - c_j \| < \varepsilon$ for $j=0,\ldots,n$.
Consider the tuple $(\pi(a_0),\ldots,\pi(a_n)) \in B^{n+1}_\sa$.
Using that $\rr(B)\leq n$, we obtain $(b_0',\ldots,b_n') \in\Lg_{n+1}(B)_\sa$ with $\| \pi(a_j) - b_j' \|< \varepsilon/2$ for $j=0,\ldots,n$.
As in the proof of \autoref{prp:rrExt}, we can lift the $b_j'$ to obtain $(b_0,\ldots,b_n) \in E^{n+1}_\sa$ satisfying
\[
\| a_j - b_j \| < \varepsilon/2, \andSep
\pi(b_j) = b_j' \quad \text{for $j=0,\ldots,n$}.
\]
Applying \autoref{prp:LnGeneralExt}, we obtain $(c_0,\ldots,c_n) \in\Lg_{n+1}(E)_\sa$ satisfying
\[
\| b_j - c_j \| < \varepsilon/2, \andSep
\pi(b_j) = \pi(c_j) \quad \text{for $j=0,\ldots,n$}.
\]
Then $(c_0,\ldots,c_n)$ has the desired properties.
\end{proof}

The following result provides the main application of the extension real rank.

\begin{thm}
\label{prp:rrExt-xrr-rr}
Let $0 \to A \to E \to B\to 0$ be an extension of \ca{s}.
Then
\[
\max \big\{ \rr(A), \rr(B) \big\}
\leq \rr(E)
\leq \max \big\{ \xrr(A), \rr(B) \big\}.
\]
Thus, if $\rr(A)=\xrr(A)$, then $\rr(E)=\max\{\rr(A),\rr(B)\}$.
\end{thm}
\begin{proof}
The left inequality follows since the real rank does not increase when passing to ideals or quotients;
see \cite[Th\'eorem\`e~1.4]{Elh95RRExt}.

To show the right inequality, set $m:=\max\{\xrr(A),\rr(B)\}$, which we may assume to be finite.
Then $\xrr(A)\leq m$, which by definition implies that $A$ has~$(\Lambda_m)$.
Applying \autoref{prp:rrExt-Ln-rr}, we obtain $\rr(E)\leq m$.
\end{proof}

\begin{cor}
Let $A$ be a \ca{} satisfying $\xrr(A)=0$. 
Then the real rank is unchanged when taking extensions with $A$, that is, if $0 \to A \to E \to B\to 0$ is an extension, then $\rr(E)=\rr(B)$.
\end{cor}

Next, we study the relationship between $(\Lambda_n)$, the real rank, and the extension real rank of a \ca.

\begin{lma}
\label{prp:rrLeqLn}
Let $A$ be a \ca.
Given $n \in \NN$, if $A$ has $(\Lambda_n)$, then $\rr(A)\leq n$.
In particular, $\rr(A)\leq\xrr(A)$.
\end{lma}
\begin{proof}
The statement is clear if $A$ is unital;
see \autoref{prp:xrrUnital}.
If $A$ is nonunital and has $(\Lambda_n)$, then we consider the extension $0 \to A \to \widetilde{A} \to \CC\to 0$.
Using that $\rr(\CC)=0\leq n$, it follows from \autoref{prp:rrExt-Ln-rr} that $\rr(\widetilde{A})\leq n$.
\end{proof}

\begin{rmk}
\label{rmk:connectionLnExt}
Let $A$ be a \ca, and let $n\in\NN$.
Consider the following statements:
\begin{enumerate}
\item
The \ca{} $A$ has $(\Lambda_n)$.
\item
Given an extension $0 \to A \to E \to B\to 0$ of \ca{s} such that $\rr(B)\leq n$, we have $\rr(E)\leq n$.
\end{enumerate}
By \autoref{prp:rrExt-Ln-rr}, (1) implies~(2).
Does the converse hold?
A natural approach would be to find for a given tuple in $M(A)^{n+1}_\sa$ with unimodular image image in $M(A)/A$ a sub-\ca{} $D\subseteq M(A)/A$ containing it and such that $\rr(D)\leq n$.
For $n=0$ this can be accomplished, which is the idea to prove \autoref{prp:charL0} below.
For $n\geq 1$ it remains unclear.
\end{rmk}

Either of the inequalities in the next result can be strict;
see Examples~\ref{exa:StableCalkin} and~\ref{exa:xrr-rrMult}.

\begin{prp}
\label{prp:rr-xrr-rrM}
Let $A$ be a \ca{}.
Then
\[
\rr(A) \leq \xrr(A) \leq \rr(M(A)).
\]
\end{prp}
\begin{proof}
The first inequality follows from \autoref{prp:rrLeqLn}.
To prove the second inequality, set $n:=\rr(M(A))$, which we may assume to be finite.
Let $\pi_A \colon M(A)\to M(A)/A$ denote the quotient map.

To verify $\xrr(A)\leq n$, let $(a_0,\ldots,a_n) \in M(A)_\sa^{n+1}$ with unimodular image in $M(A)/A$, and let $\varepsilon>0$. 
We need to find $(c_0,\ldots,c_n) \in\Lg_{n+1}(M(A))_\sa$ satisfying
\[
\| a_j - c_j \|<\varepsilon, \andSep
\pi_A(a_j) = \pi_A(c_j) \quad \text{for $j=0,\ldots,n$}.
\]

Let $(h_\lambda)_{\lambda\in\Lambda}$ be a quasicentral approximate unit for $A$ in $M(A)$.
We may assume that $0\leq h_\lambda\leq 1$ for each $\lambda$;
see for example \cite[Proposition~II.4.3.2, p.82]{Bla06OpAlgs}.

Set $P:=\prod_\Lambda M(A)$ and define
\[
J:=\big\{ (x_\lambda)_\lambda \in P : \lim_\lambda\|x_\lambda\| = 0 \big\},
\]
which is an ideal in $P$.
Set $Q:=P/J$. 
We use $[(x_\lambda)]$ to denote the image of $(x_\lambda)_\lambda\in P$ in $Q$.
We have a natural inclusion $M(A)\to P$, sending $a\in M(A)$ to the constant tuple $(a)_\lambda$.
To clarify notation, we will use bold symbols to denote elements in $Q$.
Set
\[
\tuple{h}:=[(h_\lambda)_\lambda]
\]
and note that $\tuple{h}$ commutes with every element in the image of $M(A)\to P\to Q$.

Claim:
Let $x,y\in M(A)_\sa$ with $\pi(x)\leq \pi(y)$. 
Set $\tuple{x}:=[(x)_\lambda]$ and $\tuple{y}:=[(y)_\lambda]$.
Then
\[
\tuple{x} (1-\tuple{h}) 
\leq \tuple{y} (1-\tuple{h})
\]
in $Q$.
(Note that $\tuple{x} (1-\tuple{h})$ and $\tuple{y} (1-\tuple{h})$ are self-adjoint.)

To prove the claim, choose $w\in M(A)/A$ such that $\pi(y-x)=w^*w$.
Lift $w$ to obtain $z\in M(A)$ with $\pi(z)=w$, and set $\tuple{z}:=[(z)_\lambda]$.
Then 
\[
\big\| (\tuple{x} - \tuple{y} - \tuple{z}^*\tuple{z})(1-\tuple{h}) \big\| 
= \pi(y-x-z^*z) = 0.
\]
The claim follows using that $\tuple{z}^*\tuple{z}(1-\tuple{h}) = \tuple{z}^*(1-\tuple{h})\tuple{z} \geq 0$.

The construction of the desired tuple $(z_0,\ldots,z_n)$ proceeds in several steps.

Step~1.
Set
\[
\tuple{a}_0 := [(a_0)_\lambda], \ldots,
\tuple{a}_n := [(a_n)_\lambda].
\]
Choose $\delta>0$ such that $\pi\big( \sum_j a_j^*a_j \big) \geq \delta$.
By Claim~1, we have
\[
\sum_{j=0}^n \tuple{a}_j^* \tuple{a}_j (1-\tuple{h}) \geq \delta (1-\tuple{h}).
\]

Step~2.
Using upper semicontinuity of the spectrum (\cite[Theorem~10.20]{Rud91FABook2ed}), we find $\varepsilon_1>0$ such that $\pi(x)\geq \delta$ for every $x\in M(A)_\sa$ with $\|x-2\sum_j a_j^*a_j\|<\varepsilon_1$.
Using that $\rr(M(A))\leq n$, choose $(b_0,\ldots,b_n) \in \Lg_{n+1}(M(A))_\sa$ such that $\| a_j - b_j \|$ is so small to ensure that $\| a_j - b_j \|  <\min\{\varepsilon,\delta/2\}$ for $j=0,\ldots,n$ and also
\[
\left\| \sum_{j=0}^n a_j^*b_j + \sum_{j=0}^n b_j^*a_j - 2\sum_{j=0}^n a_j^*a_j \right\|<\varepsilon_1.
\]
Set
\[
\tuple{b}_0 := [(b_0)_\lambda], \ldots,
\tuple{b}_n := [(b_n)_\lambda].
\]
By Claim~1, we have
\[
\left( \sum_{j=0}^n \tuple{a}_j^* \tuple{b}_j + \sum_{j=0}^n \tuple{b}_j^* \tuple{a}_j \right)(1-\tuple{h}) 
\geq \delta(1-\tuple{h}).
\]

Since $(b_0,\ldots,b_n)$ is unimodular, there is $\gamma>0$ such that
\[
\sum_{j=0}^n b_j^*b_j \geq \gamma.
\]

Step~3.
For $j=0,\ldots,n$ and $\lambda\in\Lambda$, set
\[
c_{j,\lambda} := h_\lambda^{1/2}b_jh_\lambda^{1/2} + (1-h_\lambda)^{1/2}a_j(1-h_\lambda)^{1/2}.
\]
We will show that $(c_{0,\lambda},\ldots,c_{n,\lambda})\in M(A)^{n+1}_\sa$ has the desired properties for $\lambda$ large enough.
It is clear that $\pi(c_{j,\lambda}) = \pi(a_j)$ for every $j=0,\ldots,n$ and for every $\lambda$.
Further, using that $(h_\lambda)_\lambda$ is quasicentral, we have
\begin{align*}
&\limsup_\lambda \| a_j - c_{j,\lambda} \| \\
&\leq \limsup_\lambda \left\| a_j h_\lambda - h_\lambda^{1/2}b_jh_\lambda^{1/2} \right\| 
+ \limsup_\lambda \left\| a_j(1-h_\lambda) - (1-h_\lambda)^{1/2}a_j(1-h_\lambda)^{1/2} \right\| \\
&\leq \|a_j-b_j\| < \varepsilon,
\end{align*}
for each $j = 0,\ldots,n$.
Thus, we may choose $\lambda_1\in\Lambda$ such that $\| a_j - c_{j,\lambda} \| <\varepsilon$ for all $j=0,\ldots,n$ and all $\lambda\geq\lambda_1$.

For $j=0,\ldots,n$ set
\[
\tuple{c_j} 
:= [(c_{j,\lambda})_\lambda]
=\tuple{b_j}\tuple{h}+\tuple{a_j}(1-\tuple{h}).
\]
It follows that
\begin{align*}
\sum_{j=0}^n \tuple{c_j}^*\tuple{c_j}
&= \sum_{j=0}^n \tuple{b_j}^*\tuple{b_j}\tuple{h}^2
+ \sum_{j=0}^n (\tuple{a_j}^* \tuple{b_j} + \tuple{b_j}^* \tuple{a_j})(1-\tuple{h})\tuple{h}
+ \sum_{j=0}^n \tuple{a_j}^*\tuple{a_j}(1-\tuple{h})^2 \\
&\geq \gamma \tuple{h}^2 + \frac{\delta}{2} 2 (1-\tuple{h})\tuple{h} + \delta(1-\tuple{h})^2 \\
&\geq \min\left\{ \gamma,\frac{\delta}{2} \right\}.
\end{align*}
Thus, $\sum_{j=0}^n \tuple{c_j}^*\tuple{c_j}$ is invertible in $Q$, which allows us to choose $\lambda\geq\lambda_1$ such that $\sum_{j=0}^n c_{j,\lambda}^* c_{j,\lambda}$ is invertible in $M(A)$.
Now the tuple $(c_{0,\lambda},\ldots,c_{n,\lambda})$ has the desired properties.
\end{proof}

\section{Permanence properties of the extension real rank}
\label{sec:Permanence}

We show that the extension real rank behaves well with respect to extensions;
see \autoref{prp:xrrExt}.
We also show that a nonseparable \ca{} has extension real rank at most $n$ if `sufficiently many' of its separable sub-\ca{s} do;
see \autoref{prp:xrrNonseparable}.
This allows us to reduce some computations of the extension real rank to the separable case.

\begin{lma}
\label{prp:LambdaExt}
Let $I$ be an ideal in a \ca{}~$A$, and let $m\in\NN$ be such that~$I$ and~$A/I$ have property~$(\Lambda_m)$.
Then $A$ has property~$(\Lambda_m)$ as well.
\end{lma}
\begin{proof}
Let $\pi_A \colon M(A)\to M(A)/A$ denote the quotient map, let $\varepsilon>0$, and let $(a_0,\ldots,a_m) \in M(A)_\sa^{m+1}$ with unimodular image in $M(A)/A$.
We need to find $(c_0,\ldots,c_m) \in \Lg_{m+1}(M(A))_\sa$ satisfying
\[
\| a_j - c_j \| < \varepsilon, \andSep
\pi_A(a_j) = \pi_A(c_j) \quad \text{for $j=0,\ldots,m$}.
\]

We consider $I$ as an ideal in $M(A)$, and we let $p\colon M(A)\to M(A)/I$ denote the quotient map.
Further, $A/I$ is an ideal in $M(A)/I$, such that $[M(A)/I]/[A/I]$ is naturally isomorphic to $M(A)/A$, and we let $q\colon M(A)/I\to M(A)/A$ denote the corresponding quotient map.
The situation is shown in the following commutative diagram with exact rows and columns:
\[
\xymatrix@R-10pt{
& & & 0 \ar[d] \\
& & & A/I \ar[d] \\
0 \ar[r]
& I \ar[r] & M(A) \ar[r]^-{p} \ar[d]^{\id}
& M(A)/I \ar[r] \ar[d]^{q}
& 0 \\
0 \ar[r]
& A \ar[r] 
& M(A) \ar[r]^-{\pi_A} 
& M(A)/A \ar[r] \ar[d]
& 0 \\
& & & 0
}.
\]

For each $j$, set $a_j' := p(a_j) \in M(A)/I$.
Then $q(a_j') = \pi_A(a_j)$, which shows that $(a_0',\ldots,a_m') \in (M(A)/I)^{m+1}_\sa$ has unimodular image in~$M(A)/A$.
Using that $A/I$ has~$(\Lambda_m)$, apply \autoref{prp:LnGeneralExt} for the extension $0 \to A/I \to M(A)/I \to M(A)/A \to 0$ to obtain $(b_0',\ldots,b_m') \in \Lg_{m+1}(M(A)/I)_\sa$ satisfying
\[
\| a_j' - b_j' \| < \varepsilon/2, \andSep
q(a_j') = q(b_j') \quad \text{for $j=0,\ldots,m$}.
\]
Next, choose $(b_0,\ldots,b_m) \in M(A)^{m+1}_\sa$ satisfying
\[
\| a_j - b_j \| <\varepsilon/2, \andSep
p(b_j) = b_j' \quad \text{for $j=0,\ldots,m$}.
\]

Then, using that the ideal $I$ has~$(\Lambda_m)$, we can apply \autoref{prp:LnGeneralExt} for the extension $0 \to I \to M(A) \to M(A)/I \to 0$ to obtain $(c_0,\ldots,c_m)\in\Lg_{m+1}(M(A))_\sa$ such that 
\[
\| b_j - c_j \| < \varepsilon/2, \andSep
p(b_j) = p(c_j) \quad \text{for $j=0,\ldots,m$}.
\]
Then $(c_0,\ldots,c_m)$ has the desired properties.
\end{proof}

\begin{prp}
\label{prp:xrrExt}
Let $I$ be an ideal in a \ca{}~$A$.
Then
\[
\xrr(A) \leq \max\big\{ \xrr(I), \xrr(A/I) \big\}.
\]
\end{prp}
\begin{proof}
Set $n:=\max\{\xrr(I),\xrr(A/I)\}$, which we may assume to be finite.
Let $m\geq n$.
By assumption, $I$ and $A/I$ have property~$(\Lambda_m)$.
Applying \autoref{prp:LambdaExt}, we get that $A$ has property~$(\Lambda_m)$.
\end{proof}

By \autoref{prp:xrrExt}, the extension real rank behave well with respect to extensions.
It is unclear if the real rank of the multiplier algebra enjoys the same permanence property:

\begin{qst}
\label{qst:rrMultExt}
Let $I$ be an ideal in a \ca{}~$A$.
Do we have
\[
\rr(M(A)) \leq \max\big\{ \rr(M(I)), \rr(M(A/I)) \big\}?
\]
\end{qst}

By \cite[Theorem~4.8(i)]{BroPed09Limits}, we have a positive answer to \autoref{qst:rrMultExt} in the case that $A$ is $\sigma$-unital and $\rr(M(I))=\rr(M(A/I))=0$.
In \cite{Thi24arX:RRMult}, we show that the answer is also positive under additional assumptions on $I$ but higher real rank of $M(A/I)$.

\begin{rmks}
\label{rmk:xrrBadBehaviour}
(1)
The real rank does not increase when passing to ideals or quotients of \ca{s}.
This is different for the extension real rank.
In particular, the extension real rank is not a dimension theory in the sense of \cite[Definition~1]{Thi13TopDimTypeI}.

To see that the extension real rank may increase when passing to quotients, let $\mathcal{B}$ denote the algebra of bounded, linear operators on a separable, infinite dimensional Hilbert space, and let $Q:=\mathcal{B}/\KK$ denote the Calkin algebra.
Then $Q\otimes\KK$ is a quotient of $\Bdd\otimes\KK$.
We have $\xrr(Q\otimes\KK)=1$ by \autoref{exa:StableCalkin}.
On the other hand, we showed in \cite{Thi24arX:RRMult} that $\xrr(\Bdd\otimes\KK)=0$.

To see that the extension real rank may increase when passing to ideals, let~$A$ be any \ca{} with $\rr(A)<\xrr(A)$, for example $Q\otimes\KK$; 
see \autoref{exa:StableCalkin}.
Then $A$ is necessarily nonunital, and we can consider $A$ as an ideal in its minimal unitization~$\widetilde{A}$.
Then $\xrr(A) > \rr(A) = \rr(\widetilde{A}) = \xrr(\widetilde{A})$.

(2)
The extension rank may increase when passing to inductive limits.
For example, consider the Calkin algbebra $Q=\Bdd/\KK$, and view $Q\otimes\KK$ as an inductive limit of $Q\otimes M_n(\CC)$.
Since $Q \otimes M_n(\CC)$ is unital, using \autoref{prp:xrrUnital} at the first step, we have
\[
\xrr(Q \otimes M_n(\CC))
= \rr(Q \otimes M_n(\CC))
= 0.
\]
On the other hand, we have $\xrr(Q\otimes\KK)=1$ by \autoref{exa:StableCalkin}.
\end{rmks}

Next, we prove upper bounds for the extension real rank of a nonseparable \ca{} by restricting to suitable separable sub-\ca{s}.

\begin{pgr}
\label{pgr:LS}
Given a \ca{} $A$, one says that a collection $\mathcal{F}$ of separable sub-\ca{s} of $A$ is \emph{$\sigma$-complete} if for every countable, upward directed subset $\mathcal{F}_0 \subseteq \mathcal{F}$ the \ca{} $\overline{\bigcup \mathcal{F}_0}$ belongs to $\mathcal{F}$.
Further, $\mathcal{F}$ is said to be \emph{cofinal} if for every separable sub-\ca{} $B_0 \subseteq A$ there exists $B \in \mathcal{F}$ with $B_0 \subseteq B$.
A $\sigma$-complete and cofinal family of separable sub-\ca{s} of $A$ is called a \emph{club};
see \cite[Section~6.2]{Far19BookSetThyCAlg}
Given clubs $\mathcal{F}_n$ of separable sub-\ca{s} of $A$ for $n\in\NN$, the intersection $\bigcap_n \mathcal{F}_n$ is also a club.

A property $\mathcal{P}$ of \ca{s} is said to satisfy the \emph{L\"{o}wenheim-Skolem condition} if for every \ca{} $A$ with $\mathcal{P}$ there exists a club of separable sub-\ca{s} of $A$ that each have $\mathcal{P}$.

Many every-day properties of \ca{s} satisfy the L\"{o}wenheim-Skolem condition.
For example, the Downwards L\"{o}wenheim-Skolem theorem (\cite[Theorem~7.1.9]{Far19BookSetThyCAlg}) shows that every axiomatizable property satisfies the L\"{o}wenheim-Skolem condition; 
and \cite[Theorem~2.5.1]{FarHarLupRobTikVigWin21ModelThy} lists 20 such properties, including `having real rank zero', `having stable rank one', and `being simple and purely infinite'.
Further, every property that is \emph{separably inheritable} in the sense of \cite[Definition~II.8.5.1]{Bla06OpAlgs} satisfies the L\"{o}wenheim-Skolem condition, and \cite[Section~II.8.5]{Bla06OpAlgs} provides many examples of separably inheritable properties, including `simplicitiy' and `vanishing $K_1$-group'.

By \cite[Proposition~4.11]{ThiVil24NowhereScattered}, nowhere scatteredness satisfies the L\"{o}wen\-heim-Skolem condition, and \cite[Lemma~3.2]{Thi23grSubhom} allows to relate properties with the L\"{o}wen\-heim-Skolem condition of ideals and quotients.

In \cite[Definition~1]{Thi13TopDimTypeI}, I introduced the concept of an abstract \emph{dimension theory} for \ca{s} as an assignment $d$ that to each \ca{} $A$ associates a number $d(A)$ such that in said definition conditions (D1)-(D6) are satisfied.
It follows from (D5) and (D6) that for each dimension theory $d$ and each $n \in \NN$, the property `$d(\cdot) \leq n$' satisfies the L\"{o}wenheim-Skolem condition;
see \cite[Paragraph~4.1]{Thi21GenRnk}.
Examples of dimension theories for \ca{s} include the real rank and the stable rank.

Thus, while it remains open if for given $n \geq 1$ the property `real rank $\leq n$' is axiomatizable (see \cite[Question~3.9.3]{FarHarLupRobTikVigWin21ModelThy}), it does at least satisfy the L\"{o}wen\-heim-Skolem condition.
\end{pgr}

\begin{thm}
\label{prp:xrrNonseparable}
Let $A$ be a (nonseparable) \ca, and let $n\in\NN$.
Assume that~$A$ contains a club of separable sub-\ca{s} $B \subseteq A$ satisfying $\xrr(B) \leq n$.
Then $\xrr(A)\leq n$.
\end{thm}
\begin{proof}
Let $\mathcal{F}$ be a $\sigma$-complete, cofinal family of separable sub-\ca{s} of $A$ with extension real rank at most $n$.
Let $\pi_A \colon M(A)\to M(A)/A$ denote the quotient map, let $(a_0,\ldots,a_n) \in M(A)_\sa^{n+1}$ with unimodular image in $M(A)/A$, and let $\varepsilon>0$. 
We need to find $(b_0,\ldots,b_n) \in\Lg_{n+1}(M(A))_\sa$ satisfying
\[
\| a_j - b_j \| < \varepsilon, \andSep
\pi_A(a_j) = \pi_A(b_j) \quad \text{for $j=0,\ldots,n$}.
\]

Set
\[
\mathcal{G} := \big\{ D \subseteq M(A) : D \text{ separable sub-\ca{}}, D \cap A \in \mathcal{F} \big\}.
\]
By \cite[Lemma~3.2(1)]{Thi23grSubhom}, $\mathcal{G}$ is a club of separable sub-\ca{s} of $M(A)$.
Since $\mathcal{G}$ is cofinal, we obtain $D \in \mathcal{G}$ containing $a_0,\ldots,a_n$.
The situation is shown in the following diagram with exact rows:

\[
\xymatrix@R-10pt{
0 \ar[r] 
& A \ar[r] \ar@{}[d]|{\subseteqRotatedUp}
& M(A) \ar[r]^-{\pi_A} \ar@{}[d]|{\subseteqRotatedUp}
& M(A)/A \ar[r] \ar@{}[d]|{\subseteqRotatedUp}
& 0 \\
0 \ar[r] 
& D \cap A \ar[r] 
& D \ar[r] 
& D/(D \cap A) \ar[r] 
& 0.
}
\]

Using that $(\pi_A(a_0),\ldots,\pi_A(a_n))$ belongs to $\Lg_{n+1}(M(A)/A)_\sa$, we can find $\delta>0$ such that $\sum_{j=0}^n \pi_A(a_j)^*\pi_A(a_j) \geq \delta$.
It follows that $D/(A \cap D)$ is a unital sub-\ca{} of $M(A)/A$ and that $(\pi_A(a_0),\ldots,\pi_A(a_n))$ belongs to $\Lg_{n+1}(D/(A \cap D))_\sa$.
Since $D \cap A \in \mathcal{F}$, we have $\xrr(D \cap A) \leq n$, and we can apply \autoref{prp:LnGeneralExt} for the extension $0 \to D \cap A \to D \to D/(D \cap A) \to 0$ to obtain $(b_0,\ldots,b_n) \in \Lg_{n+1}(D)_\sa$ satisfying
\[
\| a_j - b_j \| < \varepsilon, \andSep
\pi_A(a_j) = \pi_A(b_j) \quad \text{for $j=0,\ldots,n$}.
\]
Then $(b_0,\ldots,b_n)$, regarded as a tuple in $M(A)$, has the desired properties.
\end{proof}

\begin{rmk}
Given $n \in \NN$, it remains unclear if the property `extension real rank $\leq n$' satisfies the L\"{o}wenheim-Skolem condition.
The problem is that for a (nonseparable) \ca{} $A$, there is little connection between $M(A)$ and the multiplier algebras of the separable sub-\ca{s} of $A$.

For example, there exist nonunital, nonseparable \ca{s} $B$ with real rank zero and such that $M(B)$ is isomorphic to the minimal unitization $\widetilde{B}$;
see \cite{Sak68DerivationsSimple}, also \cite{GhaKos18ExtCpctOpsTrivMult}.
On the other hand, if $D \subseteq B$ is a nonunital, separable sub-\ca{}, then~$M(D)$ is nonseparable (\cite[Lemma~5.3]{AkeEllPedTom76DerivMultCAlg}) and therefore admits no natural inclusion into $M(B)$.

Similarly, given a nonseparable \ca{} $C$ with $\rr(M(C)) \leq n$, it is unclear if there exists a club of separable sub-\ca{s} $C' \subseteq C$ with $\rr(M(C')) \leq n$.
\end{rmk}

\section{C*-algebras with low extension real rank}
\label{sec:LowXRR}

In this section, we show that a \ca{} has property $(\Lambda_0)$ if and only if it has real rank zero and every projection in its corona algebra can be lifted to a projection in its multiplier algebra;
see \autoref{prp:charL0}.
We use this to explicate the relation between extension real rank zero and one;
see \autoref{prp:Char-xrr0}.

In the second half of the section, we present various classes of \ca{s} that have extension real rank zero or one, which leads to a computation of the real rank for extensions by such algebras;
see \autoref{prp:ComputeWithLowXRR}.
The computations of the extension real rank follow by combining our technique of reduction to the separable case from \autoref{sec:Permanence} with deep results of Brown \cite{Bro16HigherRRSR}, Lin \cite{Lin93ExpRank} and Zhang \cite{Zha91QuasidiagInterpolMultProj, Zha92RR0CoronaMultiplier1}.

\begin{lma}
\label{prp:LiftingApproxUnitProj}
Let $A$ be a \ca{} with $\rr(A)=0$, let $B \subseteq M(A)$ be a hereditary sub-\ca{}, and assume that every projection in $M(A)/A$ lifts to a projection in $M(A)$.
Let $\pi_A \colon M(A)\to M(A)/A$ denote the quotient map.
Given $b \in B$ and a projection $\bar{p} \in B/(B \cap A)$ such that $\| \pi_A(b)(1-\bar{p})\| < \varepsilon$, there exists a projection $p \in B$ lifting $\bar{p}$ and such that $\|b(1-p)\| < \varepsilon$.

In particular, $B$ has an approximate unit consisting of projections if and only if the quotient $B/(B \cap A)$ does.
\end{lma}
\begin{proof}
Clearly, if $B$ has an approximate unit of projections, then so does $B/(B \cap A)$.
The converse implication will follow from the lifting result using that every \ca{} has an approximate unit.

To verify the lifting result, let $b \in B$ and let $\bar{p} \in B/(B \cap A)$ be a projection such that $\|\pi_A(b)(1-\bar{p})\| < \varepsilon$.
By assumption, $\bar{p}$  can be lifted to a projection $p_0 \in M(A)$.
By \cite[Lemma~3.13]{BroPed91CAlgRR0}, we may arrange that $p_0$ belongs to $B$.
The element $b(1-p_0)$ belongs to $B$ and its image in the quotient by $B \cap A$ has norm $<\varepsilon$.
By properties of the quotient norm, there exists $c \in B \cap A$ such that $\| b(1-p_0) - c \| < \varepsilon$.
Then 
\[
\| b(1-p_0) - c(1-p_0) \| 
= \left\| \big( b(1-p_0) - c \big) (1-p_0) \right\|
\leq \| b(1-p_0) - c \|
< \varepsilon.
\]
Choose $\delta>0$ such that $\| b(1-p_0) - c(1-p_0) \| + \delta < \varepsilon$.

Set $C := (1-p_0)(B \cap A)(1-p_0)$, which is a hereditary sub-\ca{} of $A$ and therefore has an approximate unit of projections by \cite[Theorem~2.6]{BroPed91CAlgRR0}.
Thus, for the element $(1-p_0)c^*c(1-p_0) \in C$ we find a projection $p_1 \in C$ such that $\| (1-p_1)(1-p_0)c^*c(1-p_0)(1-p_1)\| < \delta^2$.
Then $\|c(1-p_0)(1-p_1)\| < \delta$.
Since $p_1 \leq 1-p_0$, the element $p := p_0+p_1$ is a projection in $B$ and we have $1-p = (1-p_0)(1-p_1)$.
Then
\begin{align*}
\| b(1-p) \| 
&= \| b(1-p_0)(1-p_1) \| \\
&\leq \| b(1-p_0)(1-p_1) - c(1-p_0)(1-p_1) \| + \| c(1-p_0)(1-p_1) \| \\
&\leq \| b(1-p_0) - c(1-p_0) \| + \delta
< \varepsilon,
\end{align*}
which shows that $p$ has the desired properties.
\end{proof}

\begin{thm}
\label{prp:charL0}
Let $A$ be a \ca.
Then the following are equivalent:
\begin{enumerate}
\item
The \ca{} $A$ has $(\Lambda_0)$.
\item
If $0 \to A \to E \to B\to 0$ is an extension of \ca{s} with $\rr(B)=0$, then $\rr(E)=0$.
\item
The \ca{} $A$ satisfies $\rr(A)=0$ and every projection from $M(A)/A$ can be lifted to a projection in $M(A)$.
\end{enumerate}
\end{thm}
\begin{proof}
It follows from \autoref{prp:rrExt-Ln-rr} that~(1) implies~(2).
Let $\pi\colon M(A)\to M(A)/A$ denote the quotient map.

To show that~(2) implies~(3), assume that~(2) holds.
Considering the (very) short exact sequence with $E=A$ and $B=0$, we see that $\rr(A)=0$.
Next, let $p \in M(A)/A$ be a projection.
Let $B := C^*(1,p) \subseteq M(A)/A$ denote the unital sub-\ca{} of $M(A)/A$ generated by $p$, and set $E := \pi^{-1}(B) \subseteq M(A)$.
Note that $B \cong \CC$ or $B \cong \CC\oplus\CC$, depending on whether $p=0,1$ or $p\neq 0,1$, and thus $\rr(B)=0$ in either case.
By assumption, we obtain $\rr(E)=0$.
By \autoref{prp:rrzeroExt}, projections can be lifted from quotients of real rank zero \ca{s}.
Thus, $p$ can be lifted to a projection in $E$, and thus in $M(A)$.

Let us show that~(3) implies~(1).
Note that a self-adjoint $1$-tuple is unimodular if and only if the element of the tuple is invertible.
Thus, to verify that $A$ has $(\Lambda_0)$, let $a \in M(A)_\sa$ be such that $\pi(a) \in M(A)/A$ is invertible, and let $\varepsilon>0$.
We need to find $b \in M(A)_\sa$ such that $b$ is invertible, $\pi(b)=\pi(a)$, and $\|a-b\|<4\varepsilon$.

We let $c_+$ and $c_-$ denote the positive and negative parts of a self-adjoint element~$c$, respectively.
Let $B := \overline{a_+ M(A) a_+}$ denote the hereditary sub-\ca{} of $M(A)$ generated by $a_+$.
We have $\pi(a_+) = \pi(a)_+$ and $\pi(a_-) = \pi(a)_-$, and it follows that $\pi(B)$ is the hereditary sub-\ca{} of $M(A)/A$ generated by $\pi(a)_+$; 
see also \cite[Corollary~1.5.11]{Ped79CAlgsAutGp}.

Since $\pi(a)$ is invertible, there exists $\delta>0$ such that $(-\delta,\delta) \cap \spec(\pi(a)) = \emptyset$.
Consequently, the spectra of $\pi(a)_+$ and $\pi(a)_-$ are contained in $[\delta,\infty)$.
Thus, if $f$ denotes the indicator function of $(0,\infty)$, then $\bar{p} := f(\pi(a)_+)$ is the unit of $\pi(B)$, and $1-\bar{p} = f(\pi(a)_-)$ is the unit of $\pi(\overline{a_- M(A) a_-})$.
We may assume that $\delta \leq \varepsilon$.

We have $\pi(a)_+(1-\bar{p}) = 0$, and by \autoref{prp:LiftingApproxUnitProj}, we can lift $\bar{p}$ to a projection $p \in B$ such that $\| a_+(1-p) \| < \varepsilon$.
Then $\| a_+ - a_+p \| \leq \varepsilon$, and thus $\| a_+ - pa_+ \| \leq \varepsilon$, and we obtain that
\[
\| a_+ - pa_+p \| 
\leq \| a_+ - pa_+ \| + \| pa_+ - pa_+p \| 
\leq \varepsilon + \| a_+ - a_+p \| 
< 2\varepsilon.
\]

We set
\[
b_+ := \big( pa_+p -\delta p \big)_+ + \delta p, \andSep
b_- := \big( (1-p)a_-(1-p) -\delta (1-p) \big)_+ + \delta (1-p).
\]
Then $b_+ \geq \delta p$ and $b_- \geq \delta (1-p)$, and it follows that $b := b_+ - b_-$ is invertible.
We have
\[
\| a_+ - b_+ \|
\leq \| a_+ - pa_+p \| + \| pa_+p - (pa_+p -\delta p)_+ - \delta p \|
\leq 2\varepsilon + \delta.
\]

Since $p \in B =  \overline{a_+ M(A) a_+}$ and $a_-$ is orthogonal to $a_+$, we obtain that $a_- = (1-p)a_-(1-p)$, and thus
\[
\| a_- - b_- \|
= \| (1-p)a_-(1-p) - ((1-p)a_-(1-p) -\delta (1-p))_+ - \delta (1-p) \|
\leq \delta.
\]

We deduce that
\[
\|a-b\|
\leq \| a_+ - b_+ \| + \| a_- - b_- \|
< 2\varepsilon + 2\delta \leq 4\varepsilon.
\]

Moreover, we have $\pi(a_+) \geq \delta \bar{p}$ and therefore $\big( \pi(a_+)-\delta\bar{p} \big)_+ = \pi(a_+)-\delta\bar{p}$.
It follows that
\[
\pi(b_+)
= \pi\big( (pa_+p -\delta p)_+ + \delta p \big)
= \big( \pi(a_+) -\delta \bar{p} \big)_+ + \delta\bar{p}
= \pi(a_+),
\]
and similarly $\pi(b_-) = \pi(a_-)$.
Hence, $\pi(b) = \pi(a)$, which shows that $b$ has the desired properties.
\end{proof}

\begin{rmk}
\label{rmk:charL0}
In \cite[Theorem~2.6]{Zha91QuasidiagInterpolMultProj}, Zhang provides several characterizations of the class of $\sigma$-unital \ca{s} $A$ of real rank zero such that every projection in $M(A)/A$ lifts to a projection in $M(A)$.
By \autoref{prp:charL0}, this is precisely the class of $\sigma$-unital \ca{s} with extension real rank zero.
\end{rmk}

\begin{cor}
\label{prp:SufficientL0}
Let $A$ be a \ca{} with $\rr(A)=0$ and $K_1(A)=0$.
Then $A$ has $(\Lambda_0)$.

Thus, if $\rr(A)=0$, $K_1(A)=0$ and $\xrr(A) \leq 1$, then $\xrr(A)=0$.
\end{cor}
\begin{proof}
By \cite[Corollary~3.16]{BroPed91CAlgRR0}, or \cite[Corollary~2.11]{Zha91QuasidiagInterpolMultProj}, every projection in~$M(A)/A$ can be lifted to a projection in~$M(A)$.
Now the result follows from \autoref{prp:charL0}.
\end{proof}

If $A$ is a \ca{} with real rank zero and the natural map $K_0(M(A)) \to K_0(M(A)/A)$ is surjective, then every projection in $M(A)/A$ can be lifted to a projection in~$M(A)$ by \cite[Proposition~3.15]{BroPed91CAlgRR0}.
This raises the following question:

\begin{qst}
If $A$ is a \ca{} with $(\Lambda_0)$, is the natural map $K_0(M(A)) \to K_0(M(A)/A)$ surjective?
\end{qst}

\begin{cor}
\label{prp:Char-xrr0}
Let $A$ be a \ca.
Then the following are equivalent:
\begin{enumerate}
\item
$\xrr(A)=0$;
\item
$\xrr(A)\leq 1$, $\rr(A)=0$ and every projection from $M(A)/A$ can be lifted to a projection in $M(A)$.
\end{enumerate}
If $A$ is stable, then the above conditions are also equivalent to:
\begin{enumerate}
\setcounter{enumi}{2}
\item
$\xrr(A)\leq 1$, $\rr(A)=0$ and $K_1(A)=0$.
\end{enumerate}
\end{cor}
\begin{proof}
Since $A$ satisfies $\xrr(A)=0$ if and only if $\xrr(A)\leq 1$ and $A$ has $(\Lambda_0)$, the equivalence of~(1) and~(2) follows from \autoref{prp:charL0}.
It follows from \autoref{prp:SufficientL0}, that~(3) implies~(1).

Assuming that $A$ is stable, let us shown that~(2) implies~(3).
It remains to verify $K_1(A)=0$.
Since $A$ is stable, we have $K_0(M(A))=K_1(M(A))=0$ by \cite[Theorem~10.2]{Weg93KThyBook}.
Further, since $M(A)$ is properly infinite, and hence so is $M(A)/A$, every element in $K_0(M(A)/A)$ is realized by a projection in $M(A)/A$ by \cite[Theorem~6.11.7]{Bla98KThy}.
Since every projection in~$M(A)/A$ lifts to a projection in~$M(A)$, we deduce that the natural map $K_0(M(A))\to K_0(M(A)/A)$ is surjective.
Using the six-term exact sequence in $K$-theory, we get $K_1(A) = K_0(M(A)/A) = 0$.
\end{proof}

\begin{prp}
\label{prp:ComputeWithLowXRR}
Let 
\[
0 \to A \to E \to B \to 0
\]
be an extension of \ca{s}.
Then we have the following:

\begin{enumerate}
\item
If $\xrr(A)=0$ then $\rr(E)=\rr(B)$.
\item
If $\xrr(A) \leq 1$, then with $\delta_0 \colon K_0(B)\to K_1(A)$ denoting the index map, we have
\[
\rr(E)
= \begin{cases}
\rr(B), &\text{if } \rr(B) \geq 1 \\
1, &\text{if } \rr(B) = 0 \text{ and } \rr(A)=1 \\
1, &\text{if } \rr(B) = \rr(A) = 0  \text{ and $\delta_0$ does not vanish} \\
0, &\text{if } \rr(B) = \rr(A) = 0  \text{ and $\delta_0$ vanishes}
\end{cases}
\]
\end{enumerate}
\end{prp}
\begin{proof}
This follows by combining \autoref{prp:rrExt-xrr-rr} with \autoref{prp:rrExt-rrzero}.
\end{proof}

We now turn to the problem of determining which \ca{s} have extension real rank at most one.
The main motivation is that we can compute the real rank of extensions by such algebras using \autoref{prp:ComputeWithLowXRR}.

\begin{lma}
\label{prp:xrr-from-rr-mult}
Let $n \in \NN$, and let $A$ be a \ca{} that contains a club of separable sub-\ca{s} $B \subseteq A$ satisfying $\rr(M(B)) \leq n$.
Then $\xrr(A) \leq n$.
\end{lma}
\begin{proof}
For each \ca{} $B$ in the family, we have $\xrr(B) \leq n$ by \autoref{prp:rr-xrr-rrM}.
We get $\xrr(A) \leq n$ by \autoref{prp:xrrNonseparable}.
\end{proof}

\begin{prp}
\label{prp:Classes-xrr0}
The following \ca{s} have extension real rank zero: 
\begin{enumerate}
\item
\ca{s} with real rank zero, stable rank one, and vanishing $K_1$-group.
\item
Simple, purely infinite \ca{s} with vanishing $K_1$-group.
\end{enumerate}
\end{prp}
\begin{proof}
(1)
Let $A$ be a \ca{} with $\rr(A)=0$, $\sr(A)=1$, and $K_1(A)=0$.
As noted in \autoref{pgr:LS}, the properties `real rank zero', `stable rank one', and `vanishing $K_1$-group' each satisfy the L\"{o}wenheim-Skolem condition.
Thus, there exists a club $\mathcal{F}$ of separable sub-\ca{s} $B \subseteq A$ with $\rr(B)=0$, $\sr(B)=1$, and $K_1(B)=0$.
For each $B\in\mathcal{F}$, we have $\rr(M(B))=0$ by \cite[Theorem~10]{Lin93ExpRank}.
Now we get $\xrr(A)=0$ by \autoref{prp:xrr-from-rr-mult}.

(2)
The proof is analogous to~(1), using that the properties `simple and purely infinite' and `vanishing $K_1$-group' satisfy the L\"{o}wenheim-Skolem condition, and that if $B$ is a separable, simple, purely infinite \ca{} with $K_1(B)=0$, then $\rr(M(B))=0$ by \cite[Corollary~2.6(ii)]{Zha92RR0CoronaMultiplier1}.
\end{proof}

\begin{exa}
\label{exa:xrr-rrMult}
In \cite[Section~3.3]{DowHar22ZeroDimFSpNotStrZeroDim}, it is shown that there exists a locally compact, Hausdorff space $X$ that is zero-dimensional but not strongly zero-dimensional.
This implies that $\locdim(X)=0$ while the Stone-\v{C}ech compactification $\beta X$ satisfies $\dim(\beta X)>0$.

For $A:=C_0(X)$ we have $M(A) \cong C(\beta X)$.
Note that $A$ is locally AF, and thus
\[
\rr(A)=0, \quad
\sr(A)=1, \andSep
K_1(A)=0. 
\]
Thus, by \autoref{prp:Classes-xrr0}, we have $\xrr(A)=0$.
On the other hand, we have $\rr(M(A)) = \dim(\beta X) > 0$.
This shows that the real rank of the multiplier algebra can be strictly larger than the extension real rank.
It also shows that the assumption of $\sigma$-unitality cannot be removed in Lin's theorem \cite[Theorem~10]{Lin93ExpRank}.

Using \cite[Example~3.16(i)]{Bro16HigherRRSR}, one can show that there exist separable, subhomogeneous \ca{s} $B$ such that $\xrr(B)=1<\rr(M(A))$.
\end{exa}


We refer to \cite{Ror02Classification} for details on purely infinite and $\mathcal{Z}$-stable \ca{s}.

\begin{prp}
\label{prp:Classes-xrr1}
The following \ca{s} have extension real rank $\leq 1$: 
\begin{enumerate}
\item
Simple \ca{s} with real rank zero and stable rank one. 
\item
Simple, purely infinite \ca{s}.
\item
Simple, $\mathcal{Z}$-stable \ca{s} with real rank zero.
\end{enumerate}
\end{prp}
\begin{proof}
(1)
Let $A$ be a simple \ca{} with $\rr(A)=0$ and $\sr(A)=1$.
As noted in \autoref{pgr:LS}, the properties `simplicity', `real rank zero', and `stable rank one' each satisfy the L\"{o}wenheim-Skolem condition.
Thus, there exists a club $\mathcal{F}$ of separable, simple sub-\ca{s} $B \subseteq A$ with $\rr(B)=0$ and $\sr(B)=1$.
For each $B\in\mathcal{F}$, we have $\rr(M(B)/B)=0$ by \cite[Theorem~15]{Lin93ExpRank}.
Applying \autoref{prp:rrExt}, we get
\[
\rr(M(B)) \leq \rr(B) + \rr(M(B)/B) + 1 = 1.
\]
Now the result follows from \autoref{prp:xrr-from-rr-mult}.

(2)
Let $A$ be a simple, purely infinite \ca{}.
Using that the property `simple and purely infinite' satisfies the L\"{o}wenheim-Skolem condition, we obtain a club $\mathcal{F}$ of separable, simple, purely infinite sub-\ca{s} of $A$.
Each $B\in\mathcal{F}$ is either unital or stable by \cite[Theorem~1.2(i)]{Zha92RR0CoronaMultiplier1}.
We deduce that $\rr(M(B)/B)=0$, using \cite[Corollary~2.6(i)]{Zha92RR0CoronaMultiplier1} in the stable case.
By \cite[Theorem~1.2(ii)]{Zha92RR0CoronaMultiplier1}, we also have $\rr(B) = 0$, and thus $\rr(M(B)) \leq 1$ by \autoref{prp:rrExt}.
Now the result follows from \autoref{prp:xrr-from-rr-mult}.

(3) 
Let $A$ be a simple, $\mathcal{Z}$-stable \ca{} with $\rr(A)=0$.
It follows from \cite[Theorem~4.1.10]{Ror02Classification} that $A$ is either purely infinite or stably finite.
In the first case, the result follows from ~(2).
In the second case, take any nonzero projection $p \in A$.
Then the hereditary sub-\ca{} $pAp$ is unital, simple, stably finite, and $\mathcal{Z}$-stable, and therefore has stable rank one by \cite[Theorem~6.7]{Ror04StableRealRankZ}.
It follows that $A$ has stable rank one as well, and now the result follows from~(1).
\end{proof}

The above results raise the problem to determine the possible values of the extension real rank for \ca{s} with real rank zero.
Concretely, we ask the following questions:

\begin{qst}
Does every simple \ca{} with real rank zero have extension real rank at most one?

Does every (nonsimple) \ca{} with real rank zero and stable rank one have extension real rank at most one?
\end{qst}

In fact, it is conceivable that every $\sigma$-unital \ca{} $A$ with real rank zero satisfies $\rr(M(A)) \leq 1$.

\begin{exa}
\label{exa:StableCalkin}
Let $A$ be a $\sigma$-unital, stable, simple, purely infinite \ca{} with $K_1(A)\neq 0$.
(For example, $A=Q\otimes\KK$ for the Calkin algebra $Q$.)
Then
\[
\rr(A) = 0 < 1 = \xrr(A) = \rr(M(A)).
\]

Indeed, as noted in the proof of \autoref{prp:Classes-xrr1}(2), we have $\rr(A)=0$ and $\rr(M(A)/A)=0$ by results in \cite{Zha92RR0CoronaMultiplier1}, and thus $\rr(M(A)) \leq 1$ by \autoref{prp:rrExt}.
Applying \autoref{prp:Char-xrr0} at the first step, and applying \autoref{prp:rr-xrr-rrM} at the second step, we get
\[
1 \leq \xrr(A) \leq \rr(M(A)) \leq 1.
\]
This shows that the real rank can be strictly smaller than the extension real rank.
\end{exa}

For an overview on the Corona Factorization Property, we refer to \cite{Ng06CFP}.

\begin{prp}[Brown]
\label{prp:StableCFP}
Let $A$ be a separable, stable \ca{} with an approximate unit of projections and the Corona Factorization Property.
Then $\xrr(A) \leq 1$.
\end{prp}
\begin{proof}
This is shown in the proof of \cite[Corollary~3.15]{Bro16HigherRRSR}.
\end{proof}

\begin{qst}
Does every separable, stable \ca{} with the Corona Factorization Property have extension real rank at most one?
\end{qst}

In fact, it is possible that every separable, stable \ca{} has extension real rank at most one.
A natural test case is the algebra $C(X)\otimes\KK$, where $X$ is the product of countably many copies of the two-sphere.
By \cite[Theorem~2.3]{Ng06CFP}, $C(X)\otimes\KK$ does not have the Corona Factorization Property.
It would also be interesting to compute the real rank of $M(C(X)\otimes\KK)$.

\begin{exa}
\label{exa:CFP}
By \cite[Corollary~3.5]{Rob11NuclDimComp}, $\sigma$-unital \ca{s} with finite nuclear dimension have the Corona Factorization Property.
Thus, if $A$ is a unital, separable \ca{} with finite nuclear dimension, then for every extension
\[
0 \to A\otimes\KK \to E \to B \to 0
\]
we have $\rr(E) = \rr(B)$ if $\rr(B) \geq 1$; 
we have $\rr(E)=1$ if $\rr(A)\neq 0$ and $\rr(B)=0$;
and if $\rr(A)=\rr(B)=0$ then we have $\rr(E)=0$ or $\rr(E)=1$ depending on whether the index map $K_0(B) \to K_1(A)$ vanishes.
\end{exa}

\begin{rmks}
(1)
There exist (nonunital) simple, separable \ca{s} $A$ with $\rr(A)=1$ and $\rr(M(A)/A)=0$.
In this case, we obtain
\[
\xrr(A) \leq \rr(M(A)) \leq \rr(A) + \rr(M(A)/A) + 1 = 2,
\]

For example, this applies to the stabilization of the Jiang-Su algebra $\mathcal{Z}$ by \cite[Theorem~5.8]{LinNg16CoronaStableJiangSu}, and to simple, nonunital, separable \ca{s} with strict comparison of positive elements by traces, projection injectivity and surjectivity, and quasicontinuous scale by \cite[Theorem~3.8]{Ng22RR0PICorona}.
In particular, we obtain $\xrr(\mathcal{Z}\otimes\KK) \leq 2$.
Using \autoref{prp:StableCFP}, we see that $\xrr(\mathcal{Z}\otimes\KK) = 1$.

(2)
The real rank of a \ca{s} records whether the self-adjoint $n$-tuples that generate the \ca{} as a left ideal are dense.
The \emph{generator rank} is defined analogously and records whether the self-adjoint $n$-tuples that generate the \ca{} as a \ca{} are dense.
By \cite[Proposition~3.10]{Thi21GenRnk}, the real rank of a \ca{} is dominated by its generator rank.
It is unclear if the extension real rank is dominated by the generator rank.

A unital, separable \ca{} $A$ has generator rank at most one if and only if its generators form a dense subset, that is, for every $a \in A$ and $\varepsilon>0$ there exists $b \in A$ such that $\| a - b \| < \varepsilon$ and $A = C^*(b)$;
see \cite[Remark~3.7]{Thi21GenRnk}.
It was shown in \cite{Thi22grZstableRR0, Thi23grSubhom} that every simple, nuclear, classifiable \ca{s} satisfying the Universal Coefficient Theorem has generator rank at most one.

It would also be interesting to develop a theory of `extension generator rank' in order to compute the generator rank of extensions.
\end{rmks}


\providecommand{\etalchar}[1]{$^{#1}$}
\providecommand{\bysame}{\leavevmode\hbox to3em{\hrulefill}\thinspace}
\providecommand{\noopsort}[1]{}
\providecommand{\mr}[1]{\href{http://www.ams.org/mathscinet-getitem?mr=#1}{MR~#1}}
\providecommand{\zbl}[1]{\href{http://www.zentralblatt-math.org/zmath/en/search/?q=an:#1}{Zbl~#1}}
\providecommand{\jfm}[1]{\href{http://www.emis.de/cgi-bin/JFM-item?#1}{JFM~#1}}
\providecommand{\arxiv}[1]{\href{http://www.arxiv.org/abs/#1}{arXiv~#1}}
\providecommand{\doi}[1]{\url{http://dx.doi.org/#1}}
\providecommand{\MR}{\relax\ifhmode\unskip\space\fi MR }
\providecommand{\MRhref}[2]{%
  \href{http://www.ams.org/mathscinet-getitem?mr=#1}{#2}
}
\providecommand{\href}[2]{#2}

\end{document}